\documentclass[a4poaper,11pt,reqno, english]{amsart}  
\usepackage[utf8]{inputenc}
\usepackage[T1]{fontenc}
\usepackage{amsmath,amsthm}
\usepackage{amsfonts,amssymb,enumerate}
\usepackage{url,paralist}
\usepackage{mathtools}  
\usepackage[colorlinks=true,urlcolor=blue,linkcolor=red,citecolor=magenta]{hyperref}
\usepackage{enumerate}
\usepackage{subcaption}
\usepackage{tikz}
\usepackage{tikz-cd}
\usepackage[left=1in,right=1in,top=1in,bottom=1in]{geometry}
\usepackage{float}
\theoremstyle{plain}
\newtheorem{thm}{Theorem}[section]
\newtheorem{lem}[thm]{Lemma}
\newtheorem{cor}[thm]{Corollary}
\newtheorem{prop}[thm]{Proposition}

\newtheorem*{thm*}{Theorem}

\theoremstyle{definition}
\newtheorem{defn}[thm]{Definition}
\newtheorem{ex}[thm]{Example}
\newtheorem{rem}[thm]{Remark}

\newcommand{\Z}{\mathbb{Z}}
\newcommand{\R}{\mathbb{R}}

\begin{document}

\title[Topological zero-sum Ramsey]{Topological methods in zero-sum Ramsey theory}

\author[Frick]{Florian Frick}
\address[FF]{Dept.\ Math.\ Sciences, Carnegie Mellon University, Pittsburgh, PA 15213, USA}
\email{frick@cmu.edu}

\author[Lehmann Duke]{Jacob Lehmann Duke}
\address[JL]{Dept.\ Math.\ and Stat.\, Williamstown, MA 01267, USA}
\email{jl34@williams.edu} 

\author[McNamara]{Meenakshi McNamara}
\address[MM]{Dept.\ Math.\, West Lafayette, IN 47907, USA}
\email{mcnama@purdue.edu}

\author[Park-Kaufmann]{Hannah Park-Kaufmann}
\address[HP]{Dept.\ Math.\, Bard College, Annandale-on-Hudson, NY 12504, USA}
\email{hk9622@bard.edu} 

\author[Raanes]{Steven Raanes}
\address[SR]{Dept.\ Math.\ and Stat.\, Vassar College, Poughkeepsie, NY 12604, USA}
\email{sraanes@vassar.edu} 

\author[Simon]{Steven Simon}
\address[SS]{Dept.\ Math.\, Bard College, Annandale-on-Hudson, NY 12504, USA}
\email{ssimon@bard.edu} 

\author[Thornburgh]{Darrion Thornburgh}
\address[DT]{Dept.\ Math.\, Bard College, Annandale-on-Hudson, NY 12504, USA}
\email{dt9275@bard.edu} 

\author[Wellner]{Zoe Wellner}
\address[ZW]{Dept.\ Math.\ Sciences, Carnegie Mellon University, Pittsburgh, PA 15213, USA}
\email{zwellner@cmu.edu} 

\thanks{This project was done as part of the 2023 REU: Geometry and Topology in a Discrete Setting at Carnegie Mellon University, funded by NSF CAREER Grant DMS 2042428 and NSF Grant DMS 1952285. The authors gratefully acknowledge the support of Alan Frieze. FF was supported by NSF CAREER Grant DMS 2042428. HP was supported by the Bard College Office of Undergraduate Research.}

\date{\today}

\begin{abstract}

A cornerstone result of Erd\H os, Ginzburg, and Ziv (EGZ) states that any sequence of $2n-1$ elements in $\Z/n$ contains a zero-sum subsequence of length $n$. While algebraic techniques have predominated in deriving many deep generalizations of this theorem over the past sixty years, here we introduce topological approaches to zero-sum problems which have proven fruitful in other combinatorial contexts. Our main result (1) is a topological criterion for determining when any $\Z/n$-coloring of an $n$-uniform hypergraph contains a zero-sum hyperedge. In addition to applications for Kneser hypergraphs, for complete hypergraphs our methods recover Olson's generalization of the EGZ theorem for arbitrary finite groups. Furthermore, we (2) give a fractional generalization of the EGZ theorem with applications to balanced set families and (3) provide a constrained EGZ theorem which imposes combinatorial restrictions on zero-sum sequences in the original result. 
\end{abstract}

\maketitle

\section{Introduction and Statement of Results}
\label{Introduction} 

The classical Erd\H os--Ginzburg--Ziv (EGZ) theorem~\cite{EGZ61} states that any sequence $a_1, \dots, a_{2n-1}$ of $2n-1$ elements in $\Z/n$ contains a subsequence $a_{i_1}, \dots, a_{i_n}$ with $\sum_j a_{i_j} = 0$. Over the last sixty years, this result has inspired numerous generalizations and variants, collectively known as zero-sum Ramsey theory, a general viewpoint that seems to originate from a paper of Bialostocki and Dierker~\cite{BD90}; see~\cite{Car96,GG06} for surveys. Algebraic techniques such as the Cauchy--Davenport and Chevalley--Warning theorems have proven to be particularly fruitful in deriving results of this type. Multiple different proofs of the original EGZ theorem are known (see e.g.,~\cite{AD93}); these typically proceed by first establishing the result for $n$ a prime, from which the general case follows by a simple induction on prime divisors. 

Here we introduce equivariant topological methods to the study of zero-sum Ramsey problems. Such techniques have proven to be quite powerful in other combinatorial contexts, for instance in establishing the chromatic numbers of Kneser graphs~\cite{Lov78} and hypergraphs~\cite{AFL86, Kri92}. A brief review of the necessary concepts from this well-established approach is given in Section~\ref{sec:background}. We remark that~\cite{Za20} presents an unrelated geometric approach to zero-sum Ramsey results and~\cite{KP12} develops topological methods for a problem in arithmetic combinatorics. 
In particular, we observe that known results already show that the original EGZ theorem can be seen as a  consequence of the colorful Carath\'eodory theorem~\cite{Bar82}, a central result in discrete geometry, since the latter implies a result of Drisko~\cite{Dri98} on rainbow matchings in bipartite graphs from which~\cite{AKZ18} the EGZ theorem quickly follows.

In what follows, we give three topological proofs of the Erd\H os--Ginzburg--Ziv theorem  that generalize in three distinct directions.  In particular, when $n=p$ is prime we 

\begin{compactenum}[(i)]
    \item establish a topological criterion for when any $\Z/p$-coloring of a $p$-uniform hypergraph admits a hyperedge whose labels sum to zero, by which the Erd\H os--Ginzburg--Ziv theorem is recovered in the case of a complete hypergraph; 
    \item provide a fractional generalization of the Erd\H os--Ginzburg--Ziv theorem; and
    \item prove a constrained version of the Erd\H os--Ginzburg--Ziv theorem which imposes combinatorial restrictions on zero-sum sequences in the original result. 
\end{compactenum}

We now precisely state these generalizations and collect some consequences.

\subsection{Hypergraph Coloring Generalizations of EGZ}

Let $H$ be an $n$-uniform hypergraph with vertex set~$V$, i.e., a collection of $n$-element subsets of~$V$. A \textbf{$\Z/n$-coloring} of~$H$ is a map $c\colon V \to \Z/n$, and a hyperedge $e \in H$ is said to be \textbf{zero-sum} with respect to~$c$ provided $\sum_{v \in e} c(v) = 0$. Thus an equivalent formulation of the EGZ theorem is that any $\Z/n$-coloring of the complete $n$-uniform hypergraph on $[2n-1]=\{1,2,\dots, 2n-1\}$ has a zero-sum hyperedge. 

When $n=p$ is prime, our topological criterion for when a given hypergraph~$H$ has a zero-sum hyperedge for any $\Z/p$-coloring is stated in terms of continuous maps $f\colon B(H)\rightarrow S^{2p-3}$ from the \textbf{box complex}~$B(H)$ of~$H$ to a $(2p-3)$-dimensional sphere. Box complexes have been instrumental in establishing lower bounds for the chromatic number~$\chi(H)$ of uniform hypergraphs. (Recall that $\chi(H)$ is the minimum number of colors needed to color the vertices of $H$ so that no hyperedge is monochromatic.) The complex $B(H)$ is a simplicial complex (that is, a downward-closed set system) on vertex set $V \times \Z/p$, whereby a subset $\sigma=\cup_{i\in \Z/p} (A_i\times \{i\})$ of $V\times \Z/p$ lies in $B(H)$ provided $\{a_0, \dots, a_{p-1}\} \in H$ for all $a_0 \in A_0, \dots, a_{p-1} \in A_{p-1}$. Thus for non-empty $A_i\subset V$ we have that $\sigma\in B(H)$ if and only if the $A_i$ are pairwise disjoint and $H$ contains the complete $p$-partite hypergraph determined by the $A_i$. As with any simplicial complex, one may think of $B(H)$ as a topological space glued from simplices, which in this case carries a free $\Z/p$-action that cyclically shifts the $\Z/p$-factor.

Denote the $d$-dimensional sphere by~$S^d$. Any $p$-th root of unity determines a $\Z/p$-action on $\mathbb C \cong \R^2$ given by multiplication. Considering all non-trivial roots of unity, one thereby has a free $\Z/p$-action on $S^{2p-3} \subset \R^{2p-2}$ given by the diagonal action on $\R^{2p-2} \cong (\R^2)^{p-1}$. We recall that a continuous map $f\colon X \to Y$ between two spaces $X$ and $Y$ equipped with a $\Z/p$-action is \emph{equivariant} if it commutes with the action.

 We may now state our first main result (see Section~\ref{sec:box-complex} for the proof):

\begin{thm}
\label{thm:box-complex}
    Let $p \ge 2$ be a prime, and let $H$ be a $p$-uniform hypergraph. If there is no $\Z/p$-equivariant map $B(H) \to S^{2p-3}$, then for any $\Z/p$-coloring of~$H$ there is a zero-sum hyperedge in~$H$. 
\end{thm}

The box complex of the complete $p$-uniform hypergraph on $2p-1$ vertices is equivariantly isomorphic to the $(2p-1)$-fold join of $\Z/p$ with itself, and so by an elementary Borsuk--Ulam type theorem due to Dold ~\cite{Dol83} does not admit a $\Z/p$-equivariant map to $S^{2p-3}$. Thus Theorem~\ref{thm:box-complex} recovers the Erd\H os--Ginzburg--Ziv theorem when $p$ is prime. In fact, if one replaces continuous maps with \emph{linear} ones, then the box complex construction allows for the consideration of colorings of $n$-uniform hypergraphs by any group of order $n$ (see Theorem~\ref{thm:linear box-complex} and Remark~\ref{rem:cyclic} below). Restricting to complete uniform hypergraphs and applying the ``linear Borsuk--Ulam theorem'' of Sarkaria~\cite{Sar00} gives the following extension of the EGZ theorem to arbitrary finite groups due to Olson~\cite{Ols76} as an immediate consequence (see Section~\ref{sec:box-complex} for the proof): 

\begin{thm}
\label{thm:Olson}
Let  $a_1,\ldots, a_{2n-1}$ be a sequence of elements of a group $G$ of order $n$. Then there are $n$ distinct indices $i_1,\ldots, i_n$ such that $a_{i_1}\cdots a_{i_n}=1$, the group identity. 
\end{thm}

 It follows from a result of Kr\'i\v z that any hypergraph satisfying the criterion of Theorem~\ref{thm:box-complex} must have chromatic number at least three~\cite{Kri92}. Such a lower bound for the chromatic number does not in itself guarantee a zero-sum hyperedge; see Remark~\ref{rem:chromatic} for a simple counterexample. 
Nonetheless, in the special case of Kneser hypergraphs, one can use Theorem~\ref{thm:box-complex} to give a purely combinatorial criterion for the existence of zero-sum hyperedges. Recall that for a set family $\mathcal F$ on ground set $[m] = \{1,2,\dots,m\}$, the \textbf{$n$-uniform Kneser hypergraph} $\mathrm{KG}^n(\mathcal F)$ has $\mathcal F$ as its vertex set and $A_1, \dots, A_n \in \mathcal F$ form a hyperedge if the $A_i$ are pairwise disjoint. The \textbf{$n$-colorability defect} $\mathrm{cd}^n(\mathcal F)$ of $\mathcal F$, introduced by Dolnikov for $n=2$ and in general by Kr\'i\v z, is defined by 
 \[
    \mathrm{cd}^n(\mathcal F) = m -\max\left\{\sum_{i=1}^n |A_i| \mid A_1, \dots, A_n \subset [m] \ \text{pairwise disjoint and} \ F \not\subset A_i \ \forall F \in \mathcal F \ \text{and} \ \forall i \in [n]\right\}.
\]

Kr\' i\v z proved the fundamental inequality $(n-1)\chi(\mathrm{KG}^n(\mathcal F))\geq \mathrm{cd}^n(\mathcal F)$ relating the $n$-colorability defect and the chromatic number of the Kneser hypergraph, while for $\Z/n$-colorings Theorem~\ref{thm:box-complex} will imply the following (see Section~\ref{sec:cd} for the proof):

\begin{thm}
\label{thm:cd}
    Let $n \ge 2$ be an integer, and let $\mathcal F$ be a set system with $\mathrm{cd}^n(\mathcal F) \ge 2n-1$. Then any $\Z/n$-coloring of $\mathrm{KG}^n(\mathcal F)$ has a zero-sum hyperedge. 
\end{thm} 

As a special case of Theorem~\ref{thm:cd}, let $k\geq 1$ be an integer and let $m$ be an integer satisfying $m\geq n(k+1)-1$. Considering the set system $\mathcal F$ consisting of all $k$-element subsets of~$[m]$, it is easily seen that $\mathrm{cd}^n(\mathcal F) = m - n(k-1) \ge 2n-1$. Thus Theorem~$\ref{thm:cd}$ applied to $\mathrm{KG}^n(\mathcal F)$ recovers the fact (see~\cite{BD92, Car96})
that $\mathcal F$ contains a zero-sum \emph{matching} of size~$n$ for any $\Z/n$-coloring of the hyperedges of $\mathcal F$. In particular, the Erd\H os--Ginzburg--Ziv theorem is again recovered by letting $k=1$.

\subsection{Fractional and Constrained Extensions of EGZ}

Our remaining results have for their starting point the characterization of zero-sum sequences in $\Z/n$ originally due to Marshall Hall~\cite{Hal52}:

\begin{thm}
\label{thm:Hall}
A sequence $a_1, \dots, a_n\in \Z/n$ is zero-sum if and only if there are permutations $\{b_1, \dots, b_n\}$ and $\{c_1,\ldots, c_n\}$ of $\Z/n$ with $a_i=b_i-c_i$ for all $i\in [n]$. 
\end{thm}  

We will give a topological proof of this fact when $n$ is prime; see the proof of Lemma~\ref{lem:degree} below. Thus an equivalent reformulation of the Erd\H os--Ginzburg--Ziv theorem is that any sequence in $\Z/n$ of length $2n-1$ contains a subsequence of length $n$ that is a difference of two permutations. When $n=p$ is prime, we will strengthen this reformulation of the EGZ theorem in two distinct ways.  First, we replace sequences of elements in $\Z/p$ with sequences of arbitrary probability measures on $\Z/p$. Secondly, we give additional constraints for those $a_i$ which are the difference of two permutations. 

Let $j \in \Z/p$. For any subset $A\subset \Z/p$ we denote by $A+j$ the shifted set $\{a+j \mid a \in A\}$, and likewise for any probability measure $\mu$ on~$\Z/p$, we denote by $\mu+j$ the shifted probability measure defined by $(\mu+j)(A) = \mu(A-j)$ for $A \subset \Z/p$. Our second main result is as follows (see Section~\ref{sec:chessboard} for the proof):

\begin{thm}
\label{thm:prob}
    Let $\mu_1, \dots, \mu_{2p-1}$ be a sequence of probability measures on~$\Z/p$. Then there is an injective map $\pi \colon \Z/p \to [2p-1]$ and convex coefficients $\lambda_i \ge 0$, $i \in \Z/p$, with $\sum \lambda_i = 1$ such that $\sum_{i\in \Z/p} \lambda_i(\mu_{\pi(i)} + i)$ is the uniform probability measure on~$\Z/p$.
\end{thm}

If the $\mu_i$ are Dirac measures concentrated at a single~$a_i \in \Z/p$, then realizing the uniform probability measure as a convex combination of the $\mu_{\pi(i)} + i$  requires that the $a_{\pi(i)}-i$ are pairwise distinct elements of~$\Z/p$, and so  Theorem~\ref{thm:prob} recovers the Erd\H os--Ginzburg--Ziv theorem. In the special case where each $\mu_i$ is the uniform measure on a subset $A_i\subset \Z/p$, we derive a corollary for balanced set systems which specializes to the EGZ theorem in the case that all the $A_i$ are singletons. Recall that a family of sets $\mathcal F$ is \textbf{balanced} if it admits a perfect fractional matching, that is, if there is a function $m\colon \mathcal F \to [0,1]$ such that $\sum_{A \in \mathcal F \mid v\in A} m(A)=1$ for every $v$ in the ground set of $\mathcal F$. One then has the following (see Section~\ref{sec:chessboard} for the proof):  

\begin{cor}
\label{cor:balanced}
    Let $A_1, \dots, A_{2p-1} \subset \Z/p$ be a sequence of nonempty subsets of~$\Z/p$. Then there is an injective map $\pi \colon \Z/p \to [2p-1]$ such that $A_{\pi(i)}+i$, $i \in \Z/p$, is a balanced collection of subsets. 
\end{cor}

By Theorem~\ref{thm:Hall}, the Erd\H os--Ginzburg--Ziv theorem is equivalent to the statement that for any sequence $a_1,\dots, a_{2p-1}$ of elements in $\Z/p$ there is a subsequence $a_{i_1}, \dots, a_{i_p}$ of length $p$ and pairwise distinct elements $b_1, \dots, b_p \in \Z/p$ such that $\{a_{i_1} +b_1, \dots, a_{i_p} + b_p\} = \Z/p$. Lastly, we will use our topological approach to zero-sum problems to prove restrictions on the permutation $b_1, \dots, b_p$; see Section~\ref{sec:constraints} for the proof.

\begin{thm}
\label{thm:constraints}
    Let $a_1, \dots, a_{2p-1}$ be a sequence in~$\Z/p$ and fix $d_1, \dots, d_{p-1} \in \Z/p \setminus \{0\}$. Then there is a subsequence $a_{i_1}, \dots, a_{i_p}$, $i_1 < i_2 < \dots < i_p$, and pairwise distinct $b_1, \dots, b_p \in \Z/p$ such that $\{a_{i_1} + b_1, \dots, a_{i_p} + b_p\} = \Z/p$ with the following additional constraint: If for any $j\in [p-1]$ we have that $i_j$ is odd and $i_{j+1}=i_j+1$, then we may prescribe that $b_{j+1}=b_j+d_j$.
\end{thm}

\section{Background: Equivariant Topology}
\label{sec:background}

We provide a brief introduction to the equivariant topology machinery we will apply to zero-sum Ramsey theory. The power of this machinery in applications to discrete problems is by now well-established. We refer the reader to Matou\v sek~\cite{Mat12}, de Longueville~\cite{Lon13}, and Kozlov~\cite{Koz08} for the basics.

Our approach relies on standard results for topological spaces $X$ equipped with finite group symmetries. We first collect the relevant terminology and notation.
\begin{defn} Let $G$ be a finite group with identity element $1$ and let $X$ be a topological space. 

    A $G$\textbf{-action} on $X$ is a collection $\Phi = (\phi_g)_{g \in G}$ of homeomorphisms $\phi_g\colon X \to X$ such that $\varphi_1 = \mathrm{id}_X$ and $\varphi_{g_1} \circ \varphi_{g_2} = \varphi_{g_1g_2}$ for all $g_1,g_2 \in G$. For brevity, we denote $\varphi_g(x)$ by $g\cdot x$ for any $g\in G$ and $x\in X$. 
    A $G$\textbf{-space} is a pair $(X,\Phi)$ where $X$ is a topological space and $\Phi$ is a $G$-action on~$X$. The action of $G$ on $X$ is \textbf{free} if, for any $x\in X$, we have that $g\cdot x\neq x$ whenever $g\neq 1$.
    Given two $G$-spaces $(X,\Phi)$ and $(Y, \Psi)$, a continuous map $f\colon X \to Y$ is $G$\textbf{-equivariant} if $f\circ \phi_g = \psi_g \circ f$ 
    for for all $g \in G$, or in other words that $f(g\cdot x)=g\cdot f(x)$ for all $g\in G$ and all $x\in X$, so that $f$ commutes with the respective $G$-actions. \end{defn}

A standard translation of discrete problems into the realm of topology proceeds via downward-closed set systems, also referred to as simplicial complexes, that may be interpreted as geometric objects. For our purposes it will be sufficient to restrict to finite simplicial complexes.

\begin{defn}
    A finite collection $\Sigma$ of sets is a \textbf{simplicial complex} if for every $\sigma \in \Sigma$ and $\tau \subseteq \sigma$ we have that $\tau \in \Sigma$. An element $\sigma \in \Sigma$ is called a \textbf{face} of~$\Sigma$. A singleton $\{v\} \in \Sigma$ is also called \textbf{vertex}. Let $V$ denote the set of vertices of~$\Sigma$, and denote by $e_v$, $v\in V$, the standard basis vectors of~$\R^V$. The \textbf{geometric realization}~$|\Sigma|$ of $\Sigma$ is the subspace of~$\R^V$ defined as the union of all convex hulls of faces of~$\Sigma$, where we identify $v$ with~$e_v$, that is, 
    \[
        |\Sigma| = \bigcup_{\sigma \in \Sigma} \mathrm{conv} \{e_v \mid v \in \sigma\} \subseteq \R^V.
    \]
    We will denote the geometric realization of $\Sigma$ also by $\Sigma$ if no confusion may arise, and we define the dimension $\dim \Sigma$ of a simplicial complex to be the maximum dimension of all of its faces. 
    
    Let $\Sigma$ and $T$ be simplicial complexes. A map $f\colon \Sigma \to T$ that maps vertices to vertices and such that if $\tau \subseteq \sigma \in \Sigma$ then $f(\tau) \subseteq f(\sigma)$ is a \textbf{simplicial map}. Any simplicial map $f\colon \Sigma \to T$ canonically induces a continuous map $|\Sigma| \to |T|$ by linear interpolation. When no confusion may arise, we will denote this induced map by~$f$ as well.
\end{defn}

\subsection{Borsuk-Ulam Theorems for The Deleted Join Construction} 

Given a simplicial complex $X$ and any integer $n$, its \emph{$n$-fold join} $X^{\ast n}$ is the the simplicial complex $$X^{\ast n}=\{\sigma=\cup_{i=1}^n (\sigma_i\times \{i\})\mid \sigma_1,\ldots, \sigma_n\in X\},$$ whose geometric realization consists of all convex combinations $\sum_{i=1}^n\lambda_i x_i$ where each $x_i$ lies in $\sigma_i$ for all $i\in [n]$. The subcomplex $$X^{\ast n}_\Delta=\{\sigma\in X^{\ast n}\mid \sigma_i\cap \sigma_j=\emptyset\,\, \text{for all}\,\, i\neq j\}$$ consisting of all pairwise disjoint faces of $X$ (possibly including empty faces) is called the \emph{deleted $n$-fold join}. Identifying $[n]$ with a group $G$ of order $n$ determines a free $G$-action on $X^{\ast n}$ given by left multiplication by the group.  

A particularly important case of the deleted-join construction, is when $X=\Delta_{m-1}$ is the $(m-1)$-dimensional simplex. The deleted $n$-fold join $(\Delta_{m-1})^{\ast n}_\Delta$ is easily seen to be isomorphic to the $m$-fold join  $[n]^{\ast m}$ (see, e.g., ~\cite{Mat12} for a history of its extensive applications in Tverberg-type intersection theory). Under the identification of $[n]$ with an order $n$-group $G$, one has a free $G$-action on $[n]^{\ast m}$ given by considering left multiplication on each $G$-factor of the join and extending diagonally. The isomorphism $(\Delta_{m-1})^{\ast n}_\Delta \cong [n]^{\ast m}$ is then $G$-equivariant. 

In what follows, we shall need two known results concerning  equivariant maps $(\Delta_{m-1})^{\ast n}_\Delta\rightarrow U$ from the deleted $n$-fold join of a simplex to $(m-1)$-dimensional Euclidian space $U\cong \R^{m-1}$ equipped with a linear $G$-action (i.e., $U$ is a real $(m-1)$-dimensional $G$-representation).  We say that $x\in  (\Delta_{m-1})^{\ast n}$ is a \emph{zero} of a map $f\colon (\Delta_{m-1})^{\ast n}_\Delta\rightarrow U$ provided $f(x)=0$. The first result, due to Sarkaria, concerns (affine) \emph{linear} mappings, i.e., those $f\colon (\Delta_{m-1})^{\ast n}_\Delta\rightarrow U$ which send convex combinations of the vertices of any face of the simplicial complex to the convex combinations of their images.

\begin{thm}
\label{thm:Sarkaria} Let $m,n\geq 2$ be integers and let $U$ be a real $(m-1)$-dimensional representation of a group $G$ of order $n$. If the origin is the unique element of $U$ which is fixed by the action, then any linear $G$-equivariant  map $f\colon (\Delta_{m-1})^{\ast n}_\Delta\rightarrow U$ has a zero. 
\end{thm}

We remark that the proof of Theorem~\ref{thm:Sarkaria} is entirely geometric -- while it is necessary that the action on the domain is free and that the action on the co-domain is fixed-point free away from the origin, it is the linearity assumption together with the identification of $(\Delta_{m-1})^{\ast n}$ with $[n]^{\ast m}$ that allows for a crucial application of the classical Carath\'eodory theorem~\cite{Car11} of discrete geometry (as in B\'ar\'any's colorful extension~\cite {Bar82} of that result, which Theorem~\ref{thm:Sarkaria} is equivalent to). 

The extension of Theorem~\ref{thm:Sarkaria} from \emph{linear} maps to arbitrary continuous ones holds when $G$ is elementary abelian (and, depending on the choice of representation, will be false for other groups for reasons of equivariant obstruction theory; see, e.g.,~{\cite[Remark 2.9]{Sar00}} and~\cite{Sim16}). The validity of this extension relies on the topological fact that the complex $(\Delta_{m-1})^{\ast n}_\Delta$ is $(m-2)$-connected, a fact which follows readily from its identification with the $m$-fold join $[n]^{\ast m}$. As the radial projection of $U\setminus\{0\}$ onto the unit sphere $S(U)$ is $G$-equivariant, the existence of a never-vanishing continuous map $(\Delta_{m-1})^{\ast p^k}\rightarrow U$ is therefore equivalent to the existence of a continuous equivariant map $(\Delta_{m-1})^{\ast p^k}\rightarrow S(U)$ to the corresponding representation sphere.

\begin{thm}
\label{thm:topological-delted} Let $m\geq 2$ and $k\geq 1$ be integers and let $p\geq 2$ be a prime. Suppose that $U$ is a real $(m-1)$-dimensional representation of the group $\Z/p^k$, and that the origin is the unique element of $U$ which is fixed by the action. Then there is no continuous $\Z/p^k$-equivariant map $f\colon (\Delta_{m-1})^{\ast p^k}_\Delta\rightarrow S(U)$. 
\end{thm}

 As $\{\pm 1\}^{\ast m}$ is the boundary complex of the $m$-dimensional cross-polytope, $(\Delta_{m-1})^{\ast 2}_{\Delta}$ with the $\Z/2$-action described naturally identifies with a $(m-1)$-dimensional sphere equipped with the antipodal action, so that the $\Z/2$-case of Theorem~\ref{thm:topological-delted} recovers the classical Borsuk-Ulam theorem, which in one of its many equivalent guises asserts that there is no odd continuous map $S^{m-1}\rightarrow S^{m-2}$.

\section{A box complex criterion for zero-sums}
\label{sec:box-complex}

Before proving Theorem~\ref{thm:box-complex} and its consequences, we first show how the box complex construction can be extended to colorings by an arbitrary group. Theorem~\ref{thm:box-complex} is then recovered as a strengthened special case. Let $H$ be an $n$-uniform hypergraph on vertex set~$V$, and let $G$ be any group of order~$n$ (written multiplicatively, and with identity element~$1$). As before, a $G$-coloring of the hypergraph is a map $c\colon V\rightarrow G$, and we now say that a hyperedge $e$ is zero-sum if there is some ordering $e=\{v_1,\ldots, v_n\}$ of its vertices such that $\Pi_{i=1}^n c(v_i)=1$.  

As before, let $B(H)$ be the simplicial complex on $V\times [n]$ where a subset $\cup_{i=1}^n (A_i\times \{i\})$ of $V\times [n]$ is a face of $B(H)$ if and only if $\{a_1,\ldots, a_n\}$ is a hyperedge of $H$ whenever $a_1\in A_1,\ldots a_n\in A_n$. Identifying $[n]$ with the group $G$, one has a free $G$-action on $B(H)$ arising from left multiplication of the group by defining $g\cdot (v,g')=(v,gg')$ for each vertex $(v,g')\in V\times G$ and extending the action to each face of $B(H)$.  

Consider the complex regular representation $\mathbb{C}[G]=\{\sum_{g\in G} z_g\ g\mid z_g\in \mathbb{C}\}$ of the group $G$, on which $G$ acts by linearly extending the action on $G$ afforded by left multiplication, and let $\mathbb{C}^\perp[G]=\{\sum_{g\in G} z_g\, g\mid \sum_{g\in G} z_g=0\}$, i.e., the orthogonal complement of the (trivial) diagonal subrepresentation. Thus $\mathbb{C}^\perp[G]$ is $G$-invariant, and now the origin is the only point which is fixed by every element of the group.

\begin{thm}
\label{thm:linear box-complex}
   Let $n \ge 2$ be an integer, let $H$ be an $n$-uniform hypergraph, and let $G$ be group of order $n$. If every equivariant linear map $B(H)\rightarrow \mathbb{C}^\perp[G]$ has a zero, then any $G$-coloring of $H$ admits a zero-sum hyperedge.
\end{thm}

We briefly defer the proof of Theorem~\ref{thm:linear box-complex} and its relationship to Theorem~\ref{thm:box-complex}, preferring instead to first derive Theorem~\ref{thm:Olson} as an immediate corollary. As with the EGZ theorem, the latter has the equivalent reformulation that any $G$-coloring of the complete $n$-uniform hypergraph $H$ on $[2n-1]$ has a zero-sum hyperedge. 

\begin{proof}[Proof of Theorem~\ref{thm:Olson}] Let $H$ be the complete $n$-uniform hypergraph on $[2n-1]$. Since $H$ contains any $n$-element subset of $[2n-1]$ as a hyperedge, by definition of the box complex $B(H)$ we see that $\sigma=\cup_{i=1}^n (A_i\times \{i\})$ lies in $B(H)$ if and only if the subsets $A_1,\ldots, A_n\subset [2n-1]$ are pairwise disjoint. Thus $B(H)=(\Delta_{2n-2})^{\ast n}_\Delta$ is precisely the deleted $n$-fold join of the $(2n-2)$-dimensional simplex discussed in Section~\ref{sec:background}. By Theorem~\ref{thm:Sarkaria}, any linear $G$-equivariant map $f\colon B(H)\rightarrow \mathbb{C}^\perp[G]$ must have a zero, and so by Theorem~\ref{thm:linear box-complex} any $G$-coloring of $H$ results in a zero-sum hyperedge.
\end{proof}

Our proof of Theorem~\ref{thm:linear box-complex} will rely on the following lemma, a simple consequence of Hall's Matching criterion for bipartite graphs (see, e.g.,~\cite{Hal35}, which is used in standard proofs of the Birkhoff--von Neumann theorem~\cite{Bir46}). 

\begin{lem}
\label{lem:stochastic}
    Let $A = (a_{ij})_{i,j} \in \R^{n \times n}$ be a doubly stochastic matrix, that is, all entries $a_{ij} \ge 0$ are nonnegative and the entries of each column and of each row sum to~$1$. Then there is a permutation $\pi\colon [n] \to [n]$ such that $a_{i\pi(i)} > 0$ for all $i\in [n]$.
\end{lem}

We will also need the following elementary fact.

\begin{prop}
\label{prop:trivial}   
 Let $G$ be any group, let $S=\{g_1,\ldots, g_m\}=\{h_1,\ldots, h_m\}$ be two orderings of a subset of order $m$, and let $x_i:=g_i^{-1}h_i$ for all $i\in [m]$. Then there exists a permutation $\pi$ of $[m]$ such that $\prod_{i=1}^m x_{\pi(i)}=1$.
\end{prop} 

\begin{proof}
We proceed by induction, the case $m=1$ being immediate. For $m\geq 2$, consider $x_1=g_1^{-1}h_1$. If $h_1=g_1$, then we let $S'=\{g_2,\ldots, g_m\}=\{h_2,\ldots, h_m\}$ and we are finished by induction. Supposing $h_1\neq g_1$, consider the unique $j\neq 1$ such that $g_j=h_1$. Without loss of generality, we may suppose $g_2=h_1$. Thus $x_1x_2=g_1^{-1}h_2$. If $g_1=h_2$, then we let $S''=\{g_3,\ldots, g_m\}=\{h_3,\ldots, h_m\}$ and again we are done by induction. If not, we consider the unique $j\neq 1,2$ such that $g_j=h_2$. Again without loss of generality we may suppose $j=3$, in which case we have $x_1x_2x_3=g_1^{-1}h_3$. Continuing in this fashion, we will eventually be done by induction or else, without loss of generality, we have that $h_1=g_2,h_2=g_3, \ldots,h_{m-2}=g_{m-1}$ and $h_m=g_1$. Thus $g_m=h_{m-1}$, so $x_1\cdots x_{m-1}x_m=g_1^{-1}h_{m-1}g_m^{-1}h_m=1$. 
\end{proof}

\begin{proof}[Proof of Theorem~\ref{thm:linear box-complex}]

Let $V$ denote the vertex set of~$H$ and suppose that $c \colon V \to G$ is a $G$-coloring of~$H$ without a zero-sum hyperedge. We will show there must exist a non-vanishing linear equivariant map $B(H)\rightarrow \mathbb{C}^\perp[G]$, where now we view $B(H)$ as a simplicial complex on $V\times G$. To that end, let $\mathcal{Z}$ be the set of all permutations of $G$ and let $Y$ denote the simplicial complex on $G \times G$ defined by specifying that $\sigma\in Y$ for $\sigma\subseteq G\times G$ if and only if $\sigma$ does not contain the graph $\Gamma(z)=\{(g,z(g)) \mid g \in G \}$ of any $z\in \mathcal{Z}$.

 We now define a simplicial map $f \colon B(H) \to Y$ by setting $f(v,g) = (g,g c(v))$ for $(v,g) \in V \times G$ and extending simplicially to the faces of~$B(H)$. First, we verify that our map is well-defined. Letting $\sigma \in B(H)$, we show that $f(\sigma) \subseteq G \times G$ does not contain the graph $\Gamma(z)$ of any $z\in\mathcal{Z}$. So let $\sigma=\cup_{g\in G} (A_g \times \{g\})$ be a face of $B(H)$ and let $z\colon G\rightarrow G$ be a permutation of $G$. By definition, any set $\{a_g\mid g\in G\}$ with $a_g\in A_g$ for all $g\in G$ is a hyperedge of~$H$.  If $\Gamma(z)\subset f(\sigma)$, then there would be some collection $\{a_g\}_{g\in G}$ with $a_g\in A_g$ for all $g\in G$ such that $\{g c(a_g) \colon g\in G\}=\{z(g)\colon i \in G\}$, and therefore $\{g c(a_g)\colon g\in G\}=G$. Fixing an ordering $G=\{g_1,\ldots, g_n\}$ and applying Proposition~\ref{prop:trivial} to $S=G$ and $h_i=g_ic(a_{g_i})$, we see there is some permutation $\psi$ of $[n]$ such that $\Pi_{i=1}^n c(a_{g_{\psi(i)}})=1$. This contradicts the assumption that $H$ does not contain a zero-sum hyperedge, so $\sigma$ does not contain the graph of any $z\in \mathcal{Z}$ and the map $f$ is well-defined.

As with the box complex $B(H)$, the simplicial complex $Y$ has a natural $G$-action arising from group multiplication, namely by letting $G$ act diagonally on the vertex set $G\times G$ of $Y$ and extending this action to each face of $Y$. It is then easily observed that the simplicial map $f\colon B(H)\rightarrow Y$ is equivariant with respect to the described actions. 

 We will now construct a never-vanishing linear $G$-equivariant map $Y \to \mathbb{C}^\perp[G]$. Composition then gives a never-vanishing linear equivariant map $B(H) \to Y \to \mathbb{C}^\perp[G]$, completing the proof. To that end, first consider the real regular representation $\R[G]=\{\sum_{g\in G} r_g\, g\mid r_g\in \R\}$ and let $h \colon Y \to \R[G] \times \R[G]$ be the map that is defined on vertices $(g_i,g_\ell) \in G \times G$ by $h(g_i,g_\ell) = (g_i, g_\ell)$ and otherwise interpolates linearly, which is clearly $G$-equivariant. We observe that the image of $h$ is contained in~$U$, the real $(2n-2)$-dimensional affine subspace of  $\R[G] \times \R[G]$ for which both the first and last $n$ coordinates sum to~$1$. We observe that $U$ is invariant under the $G$-action on $\R[G]\times \R[G]$. Geometrically, this action simultaneously cyclically permutes the vertices of two regular $(n-1)$-simplices lying in orthogonal $(n-1)$-dimensional subspaces. 

We now claim that $(\sum_{g\in G} \tfrac{1}{n}\,g, \sum_{g\in G} \tfrac{1}{n}\,g)$ does not lie in the image of~$h$. This claim completes the proof, since the translated function $h-(\sum_{g\in G} \tfrac{1}{n}\,g, \sum_{g\in G} \tfrac{1}{n}\,g)$ thereby defines a non-vanishing linear equivariant linear map from $Y$ to $\mathbb{C}^\perp[G]$. For contradiction, assume that there is some point $y \in Y$ with $h(y)=(\sum_{g\in G} \tfrac{1}{n}\,g, \sum_{g\in G} \tfrac{1}{n}\,g)$. The point $y$ is a convex combination of some collection of $(g_i, g_\ell)\in G\times G$, say $y = \sum_{i,\ell} \lambda_{i, \ell}\,(g_i, g_\ell)$. Since $h(y)=(\sum_i \lambda_{i, \ell} g_i, \sum_\ell \lambda_{i, \ell}g_\ell)=(\sum_{g\in G} \tfrac{1}{n}\,g, \sum_{g\in G} \tfrac{1}{n}\,g)$, the matrix $(n\cdot \lambda_{i, \ell})_{i, \ell} \in \R^{n\times n}$ is doubly stochastic. By Lemma~\ref{lem:stochastic}, there is therefore some permutation $\pi \colon [n] \to [n]$ such that $\lambda_{i, \pi(i)} > 0$ for all $i\in [n]$. As $Y$ is a simplicial complex, it must therefore contain the subset $\sigma = \{(g_i, g_{\pi(i)}) \mid i \in [n]\}$, which is the graph of a permutation of $G$, contradicting the definition of our complex $Y$. Thus the image of $h$ cannot contain $(\sum_{g\in G} \tfrac{1}{n}\,g, \sum_{g\in G} \tfrac{1}{n}\,g)$, finishing the proof. \end{proof}

We now remark on the relationship between Theorem~\ref{thm:linear box-complex} and Theorem~\ref{thm:box-complex}, as well as the distinction between chromatic number of a hypergraph and the zero-sum condition given by Theorem~\ref{thm:box-complex}. 

\begin{proof}[Proof of Theorem~\ref{thm:box-complex}]
    When $G=\mathbb{Z}/n$ is cyclic, the $\Z/n$-action on $\mathbb{C}^\perp[\Z/n]$ recovers the action on $\mathbb{C}^{n-1}$ and  $S(\mathbb{C}^\perp[\Z/n])\cong S^{2n-3}$ described in the introduction. The action on the latter is free if and only if $n=p$ is prime.
    
        Let $H$ be a $p$-uniform hypergraph such that there is a $\Z/p$-coloring of $H$ without any zero-sum hyperedge. By Theorem~\ref{thm:linear box-complex} there is a linear non-vanishing $\Z/p$-equivariant map $B(H) \to \mathbb{C}^\perp[\Z/p]$, and thus there is a $\Z/p$-equivariant map $B(H) \to S^{2p-3}$, which establishes the contrapositive of Theorem~\ref{thm:box-complex}.
\end{proof}

\begin{rem}
\label{rem:cyclic} 
 If $H$ is the complete $n$-uniform hypergraph on~$[2n-1]$, so that $B(H)=(\Delta_{2n-2})^{\ast n}_\Delta$, basic equivariant obstruction theory calculations guarantee the existence of a never-vanishing $G$-equivariant \emph{continuous} map $B(H) \rightarrow \mathbb{C}^\perp[G]$ (and therefore $G$-equivariant continuous maps $B(H)\rightarrow S(\mathbb{C}^\perp[G])\cong S^{2n-3}$) whenever $G$ is neither elementary abelian group nor isomorphic to $\Z/4$ (see, e.g.,~\cite[Remark 2.9]{Sar00}). On the other hand, any \emph{linear} equivariant map must have a zero by Theorem~\ref{thm:Sarkaria}. 
\end{rem}

\begin{rem}
\label{rem:chromatic} 
For $p$-uniform hypergraphs $H$, it follows from work of Kr\'i\v z~\cite[Theorems~2.2 and~2.6]{Kri92} 
that $\chi(H) \geq \frac{d+1}{p-1}$, where $d$ is the minimum dimension of a sphere $S^d$ equipped with a free $\Z/p$-action for which there exists a continuous $\Z/p$-equivariant map $B(H)\rightarrow S^d$. Thus the condition $d\geq 2p-2$ given by Theorem~\ref{thm:box-complex} coincides with Kr\'i\v z's criterion for $\chi(H)\geq 3$. However, the topological condition ensuring a zero-sum hyperedge given by Theorem~\ref{thm:box-complex} is stronger than the combinatorial condition that the hypergraph is not two-colorable. As a simple example, let $p=3$ and let $H$ be the 3-uniform hypergraph whose vertex set is $V = \{v_0,v_1,v_2,v_3,v_4,v_5,v_6\}$ and whose hyperedges consist of all subsets of three vertices from $V$ where exactly two vertices come from either $A = \{v_1,v_2,v_3\}$ or $B = \{v_4,v_5,v_6\}$. It is easy to see that $\chi(H) = 3$, but there is no zero-sum hyperedge corresponding to the coloring $c\colon V\to \Z/3$ defined by  $c(v_i) = 0$ for $i\in\{1,2,3\}$, $c(v_i) = 1$ for $i\in\{4,5,6\}$, and $c(v_0) = 2$. 
\end{rem}

\begin{rem}
\label{rem:Caro} 
Let $G$ be a graph with $k$ edges. Following Caro's survey on zero-sum Ramsey theory~\cite{Car96}, we denote by $R(G,2)$ the Ramsey number of~$G$, that is the smallest $n$ such that in any $2$-coloring of the edges of the complete graph~$K_n$, there is a monochromatic copy of~$G$. We let $H_n$ be the $k$-uniform hypergraph whose vertex set consists of the edges of $K_n$, with hyperedges for any $k$ edges of $K_n$ that form a copy of~$G$. Thus $R(G,2)$ is the smallest~$n$ such that $H_n$ has chromatic number at least~$3$. Now denote by $R(G,\Z/k)$ the zero-sum Ramsey number of~$G$, that is, the smallest~$n$ such that every $\Z/k$-coloring of $H_n$ has a zero-sum hyperedge. It is known that $R(G,\Z/k) \ge R(G,2)$, and that this inequality is often strict as in Remark~\ref{rem:chromatic} above. Clearly, $R(G,\Z/k) \le R(G,k)$, where $R(G,k)$ is the smallest $n$ with $\chi(H_n) > k$. While it is an open problem whether the asymptotics of $R(G,\Z/k)$ are aligned more closely with $R(G,2)$ or with~$R(G,k)$,  Theorem~\ref{thm:box-complex} does show that the \emph{topological} lower bounds on $R(G,\Z/k)$ and $R(G,2)$ are identical.
\end{rem}

\section{Proof of Theorem~\ref{thm:cd}} 
\label{sec:cd}

We now prove Theorem~\ref{thm:cd}. While this follows from Theorem~\ref{thm:box-complex} together with the work of Kr\'i\v z~\cite{Kri92}, Kr\'i\v z's construction uses a different (albeit related) notion of a box complex than our own. Instead of adapting Kr\i \v z's construction, we shall instead give a Tverberg-type intersection theory argument that allows us to obtain Theorem~\ref{thm:cd} directly from Theorems~\ref{thm:box-complex} and ~\ref{thm:topological-delted}; see~\cite{Fri20} for similar techniques. We refer the reader to Matou\v sek and Ziegler~\cite{MZ04} for an exposition of the various notions of box complexes and how they relate in the case~$p=2$. 

We give some details on the difference between Kr\i \v z's construction and our box complex. Let $H$ be a $p$-uniform hypergraph on~$V$ and denote by $B_{\mathrm{chain}}(H)$ the partially ordered set on $V \times \Z/p$, where $A_0 \times \{0\} \cup \dots \cup A_{p-1} \times \{p-1\}$ is in $B_{\mathrm{chain}}(H)$ if all the $A_i$ are nonempty and if for all $a_0 \in A_0, \dots, a_{p-1} \in A_{p-1}$ one has that $\{a_0, \dots, a_{p-1}\} \in H$. Thus $B_{\mathrm{chain}}(H)$ differs from $B(H)$ only in that it excludes those $A_0 \times \{0\} \cup \dots \cup A_{p-1} \times \{p-1\}$ where some of the $A_i$ are empty. While the set $B_{\mathrm{chain}}(H)$ is not closed under taking subsets and is therefore not in itself a simplicial complex, as a poset it has an associated order complex $C(H)$; this is the simplicial complex which Kr\' i\v z considers. We observe that this complex $C(H)$ is a subcomplex of the barycentric subdivision of our box complex ~$B(H)$; we will use this later geometric version of $B_{\mathrm{chain}}(H)$ in our proof.

\begin{lem}
\label{lem:cd-box-complex}
   Let $p \ge 2$ be a prime, let $m\geq 2$ be an integer, and suppose that $\Z/p$ acts freely on the sphere $S^{m-2}$. Suppose that $\mathcal F$ is a set system that is upwards-closed, that is, $A' \in \mathcal F$ whenever $A \in \mathcal F$ and $A' \supset A$.  
 If $\mathrm{cd}^p(\mathcal F) \ge m$, then there is no $\Z/p$-equivariant map $B(\mathrm{KG}^p(\mathcal F)) \to S^{m-2}$.
\end{lem}

\begin{proof}

Assume that the ground set of $\mathcal F$ is~$[n]$. We let $N=(p-1)(d+1)+m$, where $d\ge 0$ is chosen to be the integer such that $N\ge n$. We let $\Sigma$ be the simplicial complex on vertex set~$[N]$ defined by specifying that $\sigma \subseteq [N]$ is a face of $\Sigma$ if and only if $\sigma\notin \mathcal F$. As $\mathcal F$ is upward-closed, $\Sigma$ is downward-closed and therefore a simplicial complex. By definition of the colorability defect $\mathrm{cd}^p(\mathcal F)$, a $p$-tuple of pairwise disjoint faces of $\Sigma$ can involve at most $N-m = (p-1)(d+1)$ vertices of~$[N]$. 

We show that there exists a $\Z/p$-equivariant map $\widetilde h \colon \Sigma^{*p}_\Delta \to S^{(p-1)(d+1)-1}$ by dimension considerations. For this, let $f\colon \Sigma \to \R^d$ be a linear map that places the vertices of $\Sigma$ in generic position and interpolates affinely on the faces of~$\Sigma$. The technical notion of generiticity here, known as ``strong general position''~\cite{PS14}, is that $\bigcap_i f(\sigma_i)$ has codimension at least $\sum_i d-\dim f(\sigma_i)$ for any collection $\{\sigma_1, \dots, \sigma_p\}$ of $p$ pairwise disjoint faces of~$\Sigma$ (this is indeed generic; see~\cite{PS14}). It follows that the intersection $\bigcap_i f(\sigma_i)$ is empty for any collection of $p$ pairwise disjoint faces of~$\Sigma$, or in other words that $f$ never maps $p$ points from pairwise disjoint faces of~$\Sigma$ to the same point of~$\R^d$. Indeed, by generiticity one has $\dim f(\sigma_i) = |\sigma_i|-1$ for each $i$, and since the union of the $\sigma_i$ involves at most $(p-1)(d+1)$ vertices in all we have $\sum_i d-\dim f(\sigma_i) \ge pd - (p+1)(d+1)+p = d+1$. Now consider the map $h\colon \Sigma^{*p}_\Delta \to (\R^{d+1})^p$ defined by 
    \[
        h(\lambda_1x_1 + \dots + \lambda_px_p) = (\lambda_1, \lambda_1f(x_1), \dots, \lambda_p, \lambda_pf(x_p)).
    \]
    It follows that the image of $h$ must avoid the diagonal subset $D = \{(y_1, \dots, y_p) \in (\R^{d+1})^p \mid y_1 = y_2 = \dots = y_p\}$ of $(\R^{d+1})^p$. Indeed, if $h(\lambda_1x_1 + \dots + \lambda_px_p) \in D$ for some $\lambda_1x_1+\dots +\lambda_px_p\in  \Sigma^{*p}_\Delta$, then it follows immediately that $\lambda_1 = \lambda_2 = \dots =\frac{1}{p}$ and $f(x_1) = f(x_2) = \dots = f(x_p)$. However, the generic map $f$ cannot map $p$ points from pairwise disjoint faces to the same point, so $h$ misses~$D$. The map $h$ is $\Z/p$-equivariant if one lets $\Z/p$ act linearly on $(\R^{d+1})^p$ by cyclically permuting each $\R^{d+1}$ factor. The set $D$ is precisely the fixed point set under the action, so projecting to the orthogonal complement $D^\perp$ (which is $\Z/p$-invariant), and retracting radially to the unit sphere~$S(D^\perp)$ results in a $\Z/p$-equivariant map $\widetilde h \colon \Sigma^{*p}_\Delta \to S^{(p-1)(d+1)-1}$. 

    Now let $\mathrm{sd} ((\Delta_{N-1})^{*p}_\Delta)$ denote the barycentric subdivision of the simplicial complex $(\Delta_{N-1})^{*p}_\Delta$. We shall construct a $\Z/p$-equivariant simplicial map
    \[
        \Phi \colon \mathrm{sd} ((\Delta_{N-1})^{*p}_\Delta) \to \mathrm{sd}(\Sigma^{*p}_\Delta) * B(\mathrm{KG}^p(\mathcal F)).
    \]
    To that end, let $v$ be a vertex of $\mathrm{sd} ((\Delta_{N-1})^{*p}_\Delta)$. This vertex uniquely corresponds to a face of the complex ~$(\Delta_{N-1})^{*p}_\Delta$, or in other words to a $p$-tuple $(A_1, \dots, A_p)$ of pairwise disjoint subsets of~$[N]$. If none of the $A_i$ are in $\mathcal F$, then each $A_i$ determines a face of~$\Sigma$ and we define $\Phi(v)\in \mathrm{sd}(\Sigma^{*p}_\Delta)$ be the vertex of the barycentric subdivision of $\Sigma^{*p}_\Delta$ that corresponds to the face $A_1 * \dots * A_p$. On the other hand, suppose that  $A_i\in \mathcal F$ for at least one $i\in [p]$. Then we set $A_i=B_i$ for any such $i\in [p]$, and we set $B_i=\emptyset$ for any $i\in [p]$ with $A_i\notin \mathcal F$. We then define $\Phi(v)$ to be the vertex of the box complex $B(\mathrm{KG}^p(\mathcal F))$ that corresponds to $(B_1, \dots, B_p)$. It is now easily verified that $\Phi$ is $\Z/p$-equivariant. 

   To complete the proof, suppose that  $\Z/p$ acts freely on $S^{m-2}$ and that there is some $\Z/p$-equivariant map $B(\mathrm{KG}^p(\mathcal F)) \to S^{m-2}$. Joining this map with $\widetilde h$ above then gives a $\Z/p$-equivariant map $\Sigma^{*p}_\Delta * B(\mathrm{KG}^p(\mathcal F)) \to S^{(p-1)(d+1)-1} * S^{m-2} \cong S^{N-2}$, and composition of this map with $\Phi$ thereby yields a $\Z/p$-equivariant map $(\Delta_{N-1})^{*p}_\Delta \to S^{N-2}$. Noting that the $\Z/p$-action on the join sphere $S^{N-2}$ is free, we have reached a contradiction with Theorem~\ref{thm:topological-delted}. Thus there is no $\Z/p$-equivariant map $B(\mathrm{KG}^p(\mathcal F)) \to S^{m-2}$. 
\end{proof}

Theorem~\ref{thm:cd} now follows immediately from Lemma~\ref{lem:cd-box-complex} and Theorem~\ref{thm:box-complex} when $p$ is prime, and for arbitrary integers $n$ by a standard induction on prime divisors as in~\cite{AFL86}, ~\cite{Kri00}, and the original proof of the EGZ theorem~\cite{EGZ61}. 

\begin{proof}[Proof of Theorem~\ref{thm:cd}]

 First, let $p$ be prime. We note that we may assume that $\mathcal F$ is upward-closed. This is because if $\mathcal G$ is the set family obtained from $\mathcal F$ by including any supersets of the $A\in \mathcal F$, then it is immediate from the definition of colorability defect that $\mathrm{cd}^p(\mathcal G)\geq \mathrm{cd}^p(\mathcal F)\geq 2n-1$. Given a coloring $c\colon \mathcal F\rightarrow \Z/p$, we greedily extend it to a $\Z/p$-coloring $\widetilde{c}\colon \mathcal G\rightarrow \Z/p$. Namely, if $A_1,\ldots, A_k$ are the inclusion-maximal elements of $\mathcal{F}$, then we set $c(A_1')=A_1$ for all supersets $A_1'\supset A_1$, $c(A_2')=C(A_2)$ for all all supersets $A_2'\supset A_2$ which are not supersets of $A_1$, and so on. Considering the $\Z/p$-action on $\mathbb{C}^\perp[\mathbb{Z}/p]$ and letting $m=2p-1$ in Lemma~\ref{lem:cd-box-complex}, we have $\mathrm{cd}^p(\mathcal G)\geq 2p-1$ and so by Lemma~\ref{lem:cd-box-complex} there is no $\Z/p$-equivariant map $B(\mathrm{KG}^p(\mathcal G))\rightarrow S^{2p-3}$. Thus  $\mathrm{KG}^p(\mathcal G)$ contains a zero-sum hyperedge $\{A_1',\ldots, A_p'\}$ by Theorem~\ref{thm:box-complex}, and therefore $\mathrm{KG}^p(\mathcal F)$  contains a hyperedge $\{A_1,\ldots, A_p\}$.

For general~$n$, let $\mathcal F$ be a set system on $[m]$ with $\mathrm{cd}^n(\mathcal F)\ge 2n-1$. given any integer-valued function $c\colon \mathcal F \to \Z$, we show that there are pairwise disjoint $A_1, \dots, A_n \in \mathcal F$ such that $n$ divides $\sum_{i=1}^n c(A_i)$. As in~\cite{EGZ61}, we induct on the number of prime divisors of~$n$ and so we may assume that $n = pq$, where $p$ is a prime and $q \ge 2$ is an integer. We now follow the same reasoning as given by  Kr\'i\v z~\cite{Kri00}:  Defining $\Gamma = \{E \subseteq [m]  \mid  \mathrm{cd}^q(\mathcal F|_E) \ge 2q-1\}$, Kr\'i\v z shows that the assumption that $\mathrm{cd}^n(\mathcal F)\ge 2n-1$ implies that $\mathrm{cd}^p(\Gamma)\ge 2p-1$. Now let $E\in \Gamma$ be arbitrary. By induction, any $\Z/q$-coloring of $\mathrm{KG}^q(\mathcal F|_E)$ has a zero-sum hyperedge, that is, there are pairwise disjoint sets $A_{E,1}, \dots, A_{E,q} \in \mathcal F|_E$ such that $q$ divides $\sum_{i=1}^q c(A_{E,i})$. Defining $\widehat c \colon \Gamma \to \Z$ by $\widehat c(E) = \sum_{i=1}^q c(A_{E,i})$ and using the fact that $\mathrm{cd}^p(\Gamma)\ge 2p-1$ shows that there are pairwise disjoint $E_1, \dots, E_p \in \Gamma$ such that $p$ divides $\sum_{i=1}^p \widehat c(E_i)$. We therefore have that $n$ divides $\sum_{i=1}^p \sum_{j=1}^q c(A_{E_i,j})$, so that $\{A_{E_i,j}  \mid  i \in [p], j \in [q]\}$ is the desired zero-sum hyperedge. 
\end{proof}

\section{Zero-sums via the topology of chessboard complexes}
\label{sec:chessboard}

\subsection{Chessboard Complexes}

Our fractional generalization of the EGZ theorem and our proof of Hall's characterization of zero-sum sequences both rely crucially on the topological properties of chessboard complexes, a simplicial complex which encodes the non-attacking rook placements on an $m\times n$ chessboard. These complexes have seen extensive application to geometric combinatorics, most notably in the context of colorful extensions of Tverberg's theorem (see, e.g.,~\cite{BL92, ZV92, VZ11, BMZ15}).

\begin{defn}
[Chessboard Complex]
\label{def:chessboard}
 For positive integers $n$ and $m$, the \textbf{chessboard complex}~$\Delta_{m,n}$ is the simplicial complex where $\sigma \subseteq [m] \times [n]$ is a face of $\Delta_{m,n}$ if and only if  $(i_1, j_1),(i_2,j_2) \in \sigma$ whenever $i_1\neq i_2$ and $j_1\neq j_2$. 
\end{defn}

As an abstract simplicial complex, $\Delta_{m,n}$ consists of all matchings of the complete bipartite graph $K_{m,n}$. We also note that chessboard complexes can also be realized via the deleted join construction discussed in Section~\ref{sec:background}; indeed $\Delta_{m,n}\cong [m]^{\ast n}_{\Delta}$. For any chessboard complex $\Delta_{m,n}$, one has a $\Z/m$-action given by permuting the rows of the chessboard, as well as a free $\Z/n$-action obtained by permuting the columns.

\begin{ex}
    Consider the complex~$\Delta_{2,3}$, whose faces are 
    \begin{align*}
        &\{\{(1,1)\},\{(1,2)\},\{(1,3)\},\{(2,1)\},\{(2,2)\},\{(2,3)\}, \\
        &\{(1,1),(2,2)\},\{(1,1),(2,3)\},\{(1,2),(2,1)\},\{(1,2),(2,3)\}, \\
        &\{(1,3),(2,1)\},\{(1,3),(2,2)\}\}.
    \end{align*} 
  The geometric realization of $\Delta_{2,3}\cong [2]^{\ast 3}_\Delta\cong [3]^{\ast 2}$ is a $6$-cycle and so is homeomorphic to the circle~$S^1$.
    
    
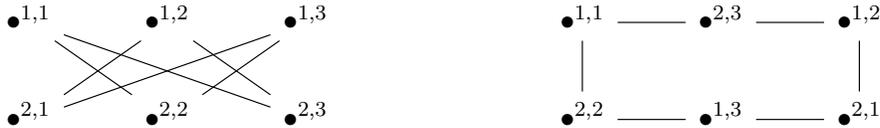
\[\begin{tikzcd}
	{\bullet^{1,1}} & {\bullet^{1,2}} & {\bullet^{1,3}} &&& {\bullet^{1,1}} & {\bullet^{2,3}} & {\bullet^{1,2}} \\
	{\bullet^{2,1}} & {\bullet^{2,2}} & {\bullet^{2,3}} &&& {\bullet^{2,2}} & {\bullet^{1,3}} & {\bullet^{2,1}}
	\arrow[no head, from=1-1, to=2-2]
	\arrow[no head, from=2-2, to=1-3]
	\arrow[no head, from=1-2, to=2-1]
	\arrow[no head, from=1-2, to=2-3]
	\arrow[no head, from=1-3, to=2-1]
	\arrow[no head, from=1-1, to=2-3]
	\arrow[no head, from=2-6, to=1-6]
	\arrow[no head, from=1-6, to=1-7]
	\arrow[no head, from=1-7, to=1-8]
	\arrow[no head, from=1-8, to=2-8]
	\arrow[no head, from=2-8, to=2-7]
	\arrow[no head, from=2-7, to=2-6]\\
\end{tikzcd}\]
\end{ex}

\subsection{A Topological Proof of Hall's Zero-Sum Criterion} 

Our proof of Theorem~\ref{thm:Hall} relies on the fact (see~\cite{BMZ15,VZ11}) that any chessboard complex of the form $\Delta_{n-1,n}$ is a $(n-2)$-dimensional orientable pseudomanifold when $n\geq 3$. In what follows, we shall not need the technical definition of such simplicial complexes, only that this property allows for a well-defined extension to pseudomanifolds of classical degree theory for maps between compact orientable manifolds (see, e.g.,~\cite{Ha00,OR09}). Considering the $\Z/n$ action on $\Delta_{n-1,n}$ induced by permuting the columns of an $(n-1)\times n$ chessboard, known results immediately give the following lemma for continuous mappings $\Delta_{p-1,p}\rightarrow \partial \Delta_{p-1}$ to the boundary complex $\partial \Delta_{p-1}$ of the $(p-1)$-simplex (which is a topological $(p-2)$-dimensional sphere).  

\begin{lem} 
\label{lem:degree}
If $p\geq 3$ is prime, then any continuous $\Z/p$-equivariant mapping $f\colon \Delta_{p-1,p}\rightarrow \partial \Delta_{p-1}$ is surjective.  
\end{lem}

\begin{proof} An equivariant extension of the classical Hopf theorem (see ~\cite{KuBa96} and ~\cite{BMZ15, VZ11}) shows that any two maps $M^d\rightarrow N^d$ between two $d$-dimensional oriented pseudomanifolds equipped with the action of a finite group $G$ have the same degree modulo $|G|$ provided the action on $M^d$ is free. It is a basic fact of degree theory that any $f\colon \Delta_{p-1,p}\rightarrow \partial \Delta_{p-1}$ which fails to be surjective must have degree zero (namely, since such a factor must factor through the sphere $S^{p-2}$ with a point deleted, which is contractible), and so the proof is complete provided there exists \emph{some} $\Z/p$-equivariant map $\Delta_{p-1,p}\rightarrow \partial \Delta_{p-1}$ whose degree is non-zero modulo~$p$. For a canonical such example (see also ~\cite{BMZ15}), consider as in ~\cite{VZ11} the natural projection $\Delta_{p-1,p}$ which sends each maximal face$\{(i_1, j_1),...,(i_{p-1}, j_{p-1})\}$ of the domain to
the maximal face $\{i_1,...,i_{p-1}\}$ of $\Delta_{p-1}$; a regular value argument now shows that the degree is precisely $(p-1)!$. \end{proof}

Lemma~\ref{lem:degree} implies the following result on partial transversals from which Hall's criterion for zero-sums in $\Z/p$ follows quickly:

\begin{thm} 
\label{thm:partialtransversal}
    Let $p \ge 2$ be a prime, and let $a_1, a_2, \dots, a_{p-1} \in \Z/p$ be a sequence of length~$p-1$. Then there are pairwise distinct $b_1, b_2, \dots, b_{p-1} \in \Z/p$ such that 
    \[
        \{a_1+b_1, a_2+b_2, \dots, a_{p-1}+b_{p-1}\} = \{0, 1, \dots, p-2\}.
    \]
\end{thm} 

\begin{proof}
    We identify the vertex set of the chessboard complex $\Delta_{p-1,p}$ with $[p-1]\times \Z/p$ and the vertex set of~$\Delta_{p-1}$ with $\Z/p$ as well (i.e., $\Delta_{p-1}$ is the standard simplex in the regular representation $\mathbb{R}[\Z/p]$). The map $f\colon \Z/p \times [p-1] \to \Z/p$ given by $f(i,j) = a_j+i$ induces a simplicial and $\Z/p$-equivariant map $f \colon \Delta_{p-1,p} \to \Delta_{p-1}$. Any such map is surjective onto $\partial\Delta_{p-1}$ by Lemma~\ref{lem:degree}, so in particular there must be a maximal face of $\Delta_{p-1,p}$ which is mapped onto the face $\{0,1,\dots, p-2\}$ of $\partial\Delta_{p-1}$. Thus there is an injective map $\pi\colon [p-1] \to \Z/p$ such that $\{0,1,\dots,p-2\} = \{f(\pi(i),i) \mid i \in [p-1]\} = \{a_{i} + \pi(i) \mid i \in [p-1]\}$, completing the proof.\end{proof}

\begin{proof}[Proof of Theorem~\ref{thm:Hall} for primes]

First, let $b=\sum_{i\in Z/p} i$ (thus $b=1$ if $p=2$ and is zero otherwise). We define the simplicial complex~$Y$ on $\Z/p\times \Z/p$ to consist of all subsets of $\Z/p \times \Z/p$ that do not contain the graph of any function $z\colon \Z/p \to \Z/p$ with $\sum_i z(i) =b$. We now let $Y'$ be the simplicial complex $Y \cup \Delta_{p,p}$ and we equip $Y'$ with the $\Z/p$-action that is defined on vertices $(i,\ell) \in \Z/p \times \Z/p$ by $j\cdot (i, \ell) = (i, \ell +j)$ for all $j \in \Z/p$.

We now show that any $\Z/p$-equivariant simplicial map from $Y' \to \Delta_{p-1}$ is surjective. To see this, first observe that the restriction of $Y'$ to $\{0,1,\dots, p-2\} \times \Z/p$ equivariantly contains~$\Delta_{p-1,p}$. By Lemma~\ref{lem:degree}, any $\Z/p$-equivariant map $\Delta_{p-1,p} \to \partial\Delta_{p-1}$ is surjective, so there is some injective map $\pi\colon\{0,1,\dots,p-2\} \to \Z/p$ so that the face $\{(i, \pi(i)) \mid i \in \{0,1,\dots, p-2\}\}$ of $\Delta_{p-1,p}$ is mapped simplicially onto the face $\{0,1,\dots,p-2\}$ of~$\partial \Delta_{p-1}$. We may now extend this face to a larger face in~$Y'$ that will surject onto~$\Delta_{p-1}$, as follows: there must be a unique vertex of the form $(p-1,j)$ of $Y'$ that maps to the vertex $p-1$ of~$\Delta_{p-1}$. We claim that $$\sigma\colon=\{(i, \pi(i)) \mid i \in \{0,1,\dots, p-2\}\} \cup \{(p-1,j)\}$$ is a face of $Y'$ and thus that $Y' \to \Delta_{p-1}$ is surjective. To see this, note that if $j \ne \pi(i)$ for all $i\in \{0,1,\dots,p-2\}$, then $\sigma$ lies in $\Delta_{p,p} \subseteq Y'$. On the other hand, if $j = \pi(i)$ for some $i \in \{0,1,\dots,p-2\}$, then we must have $j + \sum_i \pi(i) \ne b$, so that now $\sigma$ lies in $Y\subset Y'$. Thus any $\Z/p$-equivariant simplicial map from $Y'$ to $\Delta_{p-1}$ is surjective.

Finally, let $a_1, \dots ,a_p \in \Z/p$ be a given zero-sum sequence. We must show there is some permutation $\pi$ of $\Z/p$ such that $\{a_1+\pi(1),\ldots, a_p+\pi(p)\}=\Z/p$. For this, we define a simplicial map $f\colon Y' \to \Delta_{p-1}$ by specifying that $f(i,j) = a_j+i$ for each vertex $(i,j)$. Since $f$ is $\Z/p$-equivariant and is therefore surjective, there exists a function $\pi \colon \Z/p \to \Z/p$ such that $f$ carries some maximal face $\tau:=\{(i, \pi(i)) \mid i \in \Z/p\}$ of~$Y'$ onto~$\Delta_{p-1}$. To complete the proof, we show that $\tau$ is not a face of $Y$, hence that $\tau$ is a face of ~$\Delta_{p,p}$, and therefore that $\pi$ is indeed a permutation of $\Z/p$. To see this, observe that by definition $\sum_i\pi'(i)\neq b$ were $\tau$ in $Y$. However, since $\{\pi'(i)+a_i\}=\Z/p$ and $\sum_i a_i=0$ we must have $\sum_i \pi'(i)=\sum_i \pi'(i) + a_i= b$. \end{proof}

\subsection{Proof of Theorem~\ref{thm:prob}} 

We now give a second proof of the original EGZ theorem, now based on the topological connectivity of chessboard complexes (see, e.g.,~\cite{BLVZ94}). The proof method will then easily yield the probabilistic extension Theorem~\ref{thm:prob}. First, we need the following Lemma, for which we consider the free $\Z/p$-action on $\Delta_{p,2p-1}$ arising from permuting the rows of the $p\times (2p-1)$ chessboard.

\begin{lem}
\label{lem:barycenter}
If $p\geq 2$ is prime, then for any continuous $\Z/p$-equivariant map $\Delta_{p,2p-1} \to \Delta_{p-1}$ there is some point in $\Delta_{p,2p-1}$ whose image is the barycenter of $\Delta_{p-1}$.
\end{lem}

\begin{proof} As shown in ~\cite{BLVZ94}, $\Delta_{p,2p-1}$ is $(p-2)$-connected and $(p-1)$-dimensional. It follows from the discussion in Section~\ref{sec:background} immediately prior to the statement of Theorem~\ref{thm:topological-delted} that any continuous $\Z/p$-equivariant map $\Delta_{p,2p-1}\rightarrow \mathbb{R}^\perp[\Z/p]$ must have a zero. Viewing $\Delta_{p-1}$ as the standard simplex inside $\mathbb{R}[\Z/p]$, translation by the barycenter of $\Delta_{p-1}$ shows that any $\Z/p$-equivariant map $\Delta_{p,p-1}\rightarrow \Delta_{p-1}$ results in equivariant map $\Delta_{p,2p-1}\rightarrow \mathbb{R}^\perp[\Z/p]$. As such a map must have a zero, the map $\Delta_{p,p-1}\rightarrow \Delta_{p-1}$ must hit the barycenter. \end{proof}

\begin{proof}[Second Proof of the EGZ theorem for primes~$p$]
Let $a_1, \dots, a_{2p-1} \in \Z/p$ be an arbitrary sequence. 
Identifying the vertex set of the chessboard complex $\Delta_{p,2p-1}$ with $\Z/p \times [2p-1]$, we let $f\colon \Z/p \times [2p-1] \to \Z/p$ be the simplicial map defined on the vertex set by $f(i,j) = a_j+i$. This map is $\Z/p$-equivariant after identifying the vertex set of ~$\Delta_{p-1}$ with $\Z/p$. Any such map must hit the barycenter of $\Delta_{p-1}$, and since the map is simplicial there must be a face of the chessboard complex $\Delta_{p,2p-1}$ which is mapped onto~$\Delta_{p-1}$. Thus there is an injective map $\pi\colon \Z/p \to [2p-1]$ such that $\Z/p = \{f(i,\pi(i)) \mid i \in \Z/p\} = \{a_{\pi(i)} + i \mid i \in \Z/p\}$. This implies $\sum_{i\in \Z/p} i = \sum_{i \in \Z/p} a_{\pi(i)} + i$ and so that $\sum a_{\pi(i)}=0$. 
\end{proof}

Replacing simplicial maps with linear ones proves our fractional generalization of the EGZ theorem.

\begin{proof}[Proof of Thm~\ref{thm:prob}]
 First, observe that any probability measure $\mu$ on $\Z/p$ may be identified with a point $x_\mu$ in~$\Delta_{p-1}$ since both uniquely describe convex coefficients for the vertices of~$\Delta_{p-1}$. Here we again identify $\Z/p$ with the vertices of~$\Delta_{p-1}$. Explicitly, this bijective correspondence is given by $x_\mu = \sum_{i \in \Z/p} \mu(\{i\})\cdot i$. Repeating the proof of the EGZ theorem above, given a sequence $\mu_1,\ldots, \mu_{2p-1}$ of measures on $\Z/p$ we define the continuous $\Z/p$-equivariant map $f\colon \Delta_{p,2p-1} \to \Delta_{p-1}$ by setting $f(i,j) = \mu_j+ i$ on the vertices and extending to the faces of $\Delta_{p,2p-1}$ by linear interpolation. As each $\mu_j$ is a probability measure on~$\Z/p$, the image $f(i,j)$ of each vertex lies in~$\Delta_{p-1}$ and so $f$ does indeed map to $\Delta_{p-1}$. While this map is not simplicial (unless each of the $\mu_j$ are Dirac measures), Lemma~\ref{lem:barycenter} applies nonetheless and so there is a face $\sigma$ of $\Delta_{p,2p-1}$ such that $f(\sigma)$ contains the barycenter of $\Delta_{p-1}$. Thus there exists a permutation $\pi$ of $\Z/p$ and convex coefficients $\lambda_i$ such that $\sum_i \mu_{\pi (i)} + i$ equals the  barycenter of $\Delta_{p-1}$. As the barycenter of $\Delta_{p-1}$ corresponds to the uniform probability measure on $\Z/p$, the proof is complete. 
\end{proof}

Corollary~\ref{cor:balanced} immediately follows from Theorem~\ref{thm:prob}:

\begin{proof}[Proof of Cor.~\ref{cor:balanced}]
    Let $A_1, \dots, A_{2p-1} \subseteq \Z/p$ be nonempty subsets of~$\Z/p$ and associate to each $A_i$ the uniform probability distribution~$\mu_i$ supported on~$A_i$. Applying Theorem~\ref{thm:prob} to the sequence $\mu_1, \dots, \mu_{2p-1}$ thereby completes the proof.
\end{proof}

We conclude this section by showing with the following remarks concerning Lemma~\ref{lem:barycenter} and its consequences. First, it is clear from the discussion preceding Theorem~\ref{thm:topological-delted} that Theorem~\ref{thm:prob} and Corollary~\ref{cor:balanced} can be extended to arbitrary elementary abelian groups $\Z/p^k$. Secondly, as we have seen, our chessboard proof of the EGZ theorem for cyclic groups of prime order relied only on the fact that any equivariant \emph{simplicial} map $\Delta_{p,2p-1}\rightarrow \Delta_{p-1}$ hits the barycenter of the simplex, or equivalently that the simplicial map is surjective. Thus if it could be shown that any $\Z/n$-equivariant simplicial map $\Delta_{n,2n-1}\rightarrow \Delta_{n-1}$ is surjective for any integer $n\geq 2$, our proof technique would imply the EGZ theorem for arbitrary cyclic groups. As we now show, it actually follows immediately from the EGZ theorem that any $\Z/n$-equivariant simplicial map $\Delta_{n,2n-1}\rightarrow \Delta_{n-1}$ is indeed surjective. Thus the EGZ theorem and the surjectivity of such equivariant simplicial maps may be easily deduced from each other:

\begin{thm}
\label{thm:simplicial}
    Let $n \ge 2$ be an integer. Then any $\Z/n$-equivariant simplicial map $f\colon \Delta_{n,2n-1} \to \Delta_{n-1}$ is surjective.
\end{thm} 

\begin{proof} As before, we identify the vertex set of $\Delta_{n-1}$ by $\Z/n$. By the remarks above, we only need to show that the EGZ theorem implies that a given $\Z/n$-equivariant simplicial map $f\colon \Delta_{n,2n-1}\rightarrow \Delta_{n-1}$ hits the barycenter of $\Delta_{n-1}$. To that end, consider the sequence $a_1=f(0,1),\ldots, a_{2n-1}={f(0,2n-1)}$ in $\Z/n$. This has a zero-sum subsequence of length $n$ by the EGZ theorem, and thus there is a permutation $\pi$ of $\Z/n$ such that $\sum_i f(0,\pi (i))=0$. The set $\{(i,\pi(i)) \mid i\in \Z/p\}$ is a face of $\Delta_{n,2n-1}$, and by equivariance we have that $f(i,\pi(i))=a_{\pi (i)}+i$. Letting $x=\sum_i \frac{1}{n} (i,\pi (i))$, we now have that $f(x)=\sum_i \frac{1}{n} f(i,\pi (i))= \sum_i \frac{1}{n}a_{\pi(i)}+ \sum_i\frac{1}{n} i= \sum_i \frac{1}{n}i$ is the barycenter of $\Delta_{n-1}$. 
\end{proof}

\section{Erd\H os--Ginzburg--Ziv plus constraints}
\label{sec:constraints}

We conclude with the proof of our constrained version of the EGZ theorem. 

\begin{proof}[Proof of Thm.~\ref{thm:constraints}]
    Let $a_1, \dots, a_{2p-1} \in \Z/p$. We identify the vertex set of $(\Delta_{2p-2})^{*p}_\Delta$ with $\Z/p \times [2p-1]$. Suppose now that  $X \subseteq (\Delta_{2p-2})^{*p}_\Delta$ is a $\Z/p$-equivariant $(2p-3)$-connected subcomplex of $(\Delta_{2p-2})^{*p}_\Delta$. For any $i \in \Z/p$, we let $Y_i$ be the subcomplex of $(\Delta_{2p-2})^{*p}_\Delta$ defined by letting $\sigma = \{0\} \times A_0 \cup \dots \cup \{p-1\} \times A_{p-1}$ be a face of~$Y_i$ if $|A_i| \le 1$. Thus $\cap_{i\in \Z/p} Y_i=\Delta_{p,2p-1}$ is precisely the $p\times (2p-1)$ chessboard complex. Letting $d\colon X\times X \to [0, \infty)$ be any metric on $X$ compatible with the topology of $X$ as a simplicial complex (e.g., the $\ell_1$-metric), for each $x \in X$ we denote by $d(x, Y_i)= \min_{y \in Y_i} d(x,y_i)$ the distance of $x$ to the subcomplex~$Y_i$. We thus have $x\in\Delta_{p,2p-1}$ if and only if $d(x,Y_i)=0$ for all $i\in \Z/p$.

    As in our second proof of the EGZ theorem, we define $f\colon \Z/p \times [2p-1] \to \Z/p$ by $f(i,j) = a_j +i$, which induces a $\Z/p$-equivariant simplicial map $f\colon X \to \Delta_{p-1}$. Thinking of $\Delta_{p-1}$ as the standard simplex in $\mathbb{R}[\Z/p]$ and letting $b$ denote its barycenter, we now define 
    \[
        F\colon X \to (\mathbb{R}^\perp[\Z/p])^{\oplus 2}, \ x \mapsto (f(x)-b, d(x,Y_0)-a(x), \dots, d(x,Y_{p-1})-a(x)),
    \] where $a(x)=\frac{1}{p}\sum_i d(x,Y_i)$ is the average distance of $x\in X$ to the $Y_i$. 
    
    The map $F$ is $\Z/p$-equivariant, and since by assumption $X$ is $(2p-3)$-connected, it follows from Theorem~\ref{thm:topological-delted} and the remarks preceding it at that $F$ must have a zero~$x$. Thus $f(x)$ is at the barycenter of~$\Delta_{p-1}$, and moreover we have that $d(x,Y_0) = \dots = d(x, Y_{p-1})=a(x)$. Let $\sigma = \{0\} \times A_0 \cup \dots \cup \{p-1\} \times A_{p-1}$ be the inclusion-minimal face of $X$ that contains~$x$. We now claim that $a(x)=0$, so that $d(x,Y_i)=0$ for all $i\in Z/p$ and $x\in \Delta_{p,2p-1}$. Indeed, if $a(x)>0$, then  $d(x, Y_j) > 0$ for some $j\in \Z/p$ and therefore that $d(x,Y_i)>0$ for all $i\in \Z/p$. We would therefore have that $|A_i|>1$ for all $i\in \Z/p$. Since by definition of~$(\Delta_{2p-1})^{*p}_\Delta$ the $A_j$ are pairwise disjoint, this implies that $|\bigcup_j A_j| \ge 2p$, a contradiction since $\bigcup_j A_j \subseteq [2p-1]$. Thus $|A_j| \le 1$ for all~$j$. On the other hand, since $f(\sigma) = \Delta_{p-1}$, we therefore must have $|A_j| = 1$ for all~$j$. Thus $\sigma = \{(i, \pi(i)) \mid i \in \Z/p\}$ for some injective map $\pi\colon \Z/p \to [2p-1]$, and since $\{f(i, \pi(i)) \mid i \in \Z/p\} = \Z/p$ we have that $\sum a_{\pi(i)} = 0$ as before.

    To finish the proof, we therefore only need to verify that the constraints of Theorem~\ref{thm:constraints} give rise to a $(2p-3)$-connected subcomplex $X \subseteq (\Delta_{2p-2})^{*p}_\Delta$. The subcomplex $X$ is defined by the property that for every face $\sigma$ of~$(\Delta_{2p-2})^{*p}_\Delta$, we have that whenever $(i,2j-1) \in \sigma$ then $(i+x, 2j) \notin \sigma$ unless $x = d_j$. The complex $X$ is indeed highly connected, since it is the join of $(p-1)$ circles (corresponding to $\{a_1,a_2\}, \dots, \{a_{2p-3},a_{2p-2}\})$ and $p$ discrete points (corresponding to~$a_{2p-1}$). Restricting the vertex set of $X$ to $\Z/p \times \{2j-1,2j\}$ yields a $\Z/p$-equivariant cycle of length $2p$ that traverses $(0,2j-1),(d_j,2j),(d_j,2j-1), (2d_j, 2j), \dots, (pd_j, 2j)$. This closes up only after $2p$ steps because $p$ is prime. As there are no further constraints, $X$ is the join of these cycles and the $p$ vertices corresponding to the restriction of $X$ to $\Z/p \times \{2p-1\}$. As the $(p-1)$-fold join of the circle $S^1$ is a $(2p-3)$-dimensional sphere $S^{2p-3}$, we see that $X=\cup_{i\in \Z/p} D_i^{2p}$ is the union of the $p$ cones $D_i^{2p-2}=S^{2p-3}\ast \{i\}$, each of which is a $(2p-2)$-dimensional disk. Thus $X$ is $(2p-3)$-connected. \end{proof}

\bibliography{bib}{}

\begin{thebibliography}{}

\bibitem{AKZ18} R. Aharoni, D. Kotlar, and R. Ziv. Uniqueness of the extreme cases in theorems of Drisko and Erd\H os--Ginzburg--Ziv, \emph{Eur. J. Comb.}, Vol. 67 (2018) 222--229.

\bibitem{AD93} N. Alon and M. Dubiner. Zero-sum sets of prescribed size, in \emph{Combinatorics, Paul Erd\H os is Eighty, Vol. 1, Keszthely (Hungary),} Bolyai Soc. Math. Stud., 33--50, Janos Bolyai Math. Soc., Budapest, 1993.

\bibitem{AFL86} N. Alon, P. Frankl, and L. Lova\'sz. The chromatic number of Kneser hypergraphs. \textit{Trans. Amer. Math. Soc.}, Vol. 298 (1986) 359--370.

\bibitem{Bar82}  I. B\'ar\'any. A generalization of Carath\'eodory's theorem. \textit{Discrete Math.}, Vol. 40, No. 2--3 (1982) 141--152. 

\bibitem{BL92} I. B\'ar\'any and D. G. Larman. A colored version of Tverberg's theorem. \textit{J. London Math. Soc.} Vol. 45, No. 2 (1992), 314--320.

\bibitem{BD92} A. Bialostocki and P. Dierker. On the Erd\H os--Ginzburg--Ziv theorem and the Ramsey numbers for stars and matchings. \textit{Discrete Math.}, Vol. 110, No. 1--3 (1992) 1--8.

\bibitem{BD90} A. Bialostocki and P. Dierker. Zero-sum Ramsey theorems. \textit{Congress. Numer.}, Vol. 70 (1990) 119--130.

\bibitem{Bir46} G. Birkhoff. Three observations on linear algebra. \emph{Univ. Nac. Tacum\'an Rev. Ser. A}, Vol. 5 (1946) 147--151.

\bibitem{BLVZ94} A. Bj\"orner, L. Lov\'asz, S.T. Vre\'cica, and R.T. \v Zivaljevi\'c, Chessboard complexes and matching complexes, \emph{J. London Math. Soc.}, Vol. 49, No. 1 (1994) 25--39.

\bibitem{BMZ15} P. V. M. Blagojev\'ic, B. Matschke, and G. M. Ziegler. Optimal bounds for the colored Tverberg problem. \emph{J. Eur. Math. Soc.} Vol. 17, No. 4 (2015) 739--754.

\bibitem{Car11} C. Carath\'eodory. \"Uber den ariabilit\"atsbereich der Fourier'schen konstanten von positiven harmonischen funktionen. \emph{Rendiconti del Circolo Matematico di Palermo (1884?1940)}, Vol. 32 (1911) 193--217.

\bibitem{Car91} Y. Caro. On zero-sum delta systems and multiple copies of hypergraphs. \textit{J. Graph Theory}, Vol. 15 (1991) 511--521.

\bibitem{Car96} Y. Caro. Zero-sum problems -- A survey. \textit{Discrete Math.}, Vol. 152 (1996), 93--113.

\bibitem{Dol83} A. Dold. Simple proofs of some Borsuk-Ulam results, Proc. Northwestern Homotopy Theory Conf. (H. R. Miller and S. B. Priddy, eds.), \emph{Contemp. Math.}, Vol. 19 (1983), 65?69.

\bibitem{Dri98} A. A. Drisko, Transversals in row-Latin rectangles, \emph{J. Combin. Theory Ser. A}, Vol. 84 (1998) 181--195.

\bibitem{EGZ61} P. Erd\H os, A. Ginzburg, and A. Ziv. A theorem in additive number theory, \emph{Israel Research and Development Nat. Council Bull., Sect. F}, Vol. 10 (1961), 41--43.

\bibitem{Fri20} F. Frick. Chromatic numbers of stable Kneser hypergraphs via topological Tverberg-type theorems, \textit{Int. Math. Res. Not. IMRN}, Vol. 13 (2020) 4037--4061. 

\bibitem{GG06} W. Gao and A. Geroldinger. Zero-sum problems in finite abelian groups: A survey, \emph{Expo. Math.}, Vol. 24, No. 4 (2006) 337--369. 

\bibitem{Hal52} M. Hall. A combinatorial problem on abelian groups. \emph{Proc. Amer. Math. Soc.} Vol. 3, No. 4 (1952) 584--587. 

\bibitem{Hal35} P. Hall. On representatives of subsets, \emph{J. London Math. Soc.}, Vol. 10, No. 1 (1935) 26--30.

\bibitem{Ha00} A. Hatcher. \textit{Algebraic Topology}, Cambridge University Press. Cambridge (2000). 

\bibitem{KP12} R. Karasev and F. Petrov. Partitions of nonzero elements of a finite field into pairs. \emph{Israel J. Math.}, Vol. 192, No. 1 (2012) 143--156.

\bibitem{Koz08} D. Kozlov. \emph{Combinatorial Algebraic Topology}. Springer Berlin, Heidelberg (2008). 

\bibitem{Kri00} I. Kr\'i\v z. A correction to ``Equivariant cohomology and lower bounds for chromatic numbers''. \emph{Trans. Amer. Math. Soc.} Vol. 352 (2000) 1951--1952.

\bibitem{Kri92} I. Kr\'i\v z. Equivariant cohomology and lower bounds for chromatic numbers. \emph{Trans. Amer. Math. Soc.}, Vol. 333, No. 2 (1992) 567--577. 

\bibitem{KuBa96} A. Kushkuley and Z. Balanov. \textit{Geometric Methods in Degree Theory for Equivariant Maps, Lecture Notes in Math.}, Vol. 1632, Springer, Berlin (1996).

\bibitem{Lon13} M. de Longueville. \textit{A Course in Topological Combinatorics}. Universitext. New York, NY. Springer (2013).

\bibitem{Lov78} L. Lov\'asz. Kneser's conjecture, chromatic number and homotopy. \emph{J. Combin. Theory Ser. A}, Vol. 25 (1978) 319--324.

\bibitem{Mat12} J. Matou\v sek. \textit{Lectures on Discrete Geometry}, Graduate Texts in Mathematics Vol. 212, Springer Science \& Business Media (2013). 

\bibitem{MZ04} J. Matou\v sek and G. M. Ziegler. Topological lower bounds for the chromatic number: A hierarchy. \emph{Jahresber. Deutsch. Math.-Verein.}, Vol. 106, No. 2 (2004) 71--90.

\bibitem{Ols76} J. E. Olson. On a combinatorial problem of Erd\H os, Ginzburg, and Ziv. \emph{J. Number Theory}, Vol. 8, No. 1 (1976) 52--57.

\bibitem{OR09} E. Outerelo and J. M. Ruiz. \textit{Mapping degree theory}, Vol. 108, American Mathematical Society (2009).

\bibitem{PS14} M. A. Perles and M. Sigron. Strong general position. arXiv:1409.2899 [math.CO]

\bibitem{Sar00} K. S. Sarkaria. Tverberg partitions and Borsuk--Ulam theorems. \textit{Pacific J. Math.}, Vol. 196, No. 1 (2000) 231--241.

\bibitem{Sim16} S. Simon. Average-value Tverberg Partitions via Finite Fourier Analysis. \textit{Israel J. Math.}, Vol. 216, No. 2 (2016) 891--904. 

\bibitem{Tve66} H. Tverberg. A generalization of Radon's Theorem. \textit{J. London Math. Soc.}, Vol. 41 (1966) 123--128.

\bibitem{Vol96} A. Yu. Volovikov. On a topological generalization of the Tverberg theorem. \textit{Math. Notes}, Vol. 59, No. 3 (1996) 324--326.

\bibitem{VZ11} S. T. Vre\'cica and R. T. \v Zivaljevi\'c. Chessboard complexes indomitable. \emph{J. Combin. Theory Ser. A}, Vol. 118, No. 7 (2011) 2157--2166. 

\bibitem{Za20} D. Zakharov. Convex geometry and the Erd\H os--Ginzburg--Ziv problem. arXiv:2002.09892 [math.CO]

\bibitem{ZV92} R. T. \v Zivaljevi\'c and S. T. Vre\'cica. The colored Tverberg's problem and complexes of injective functions, \emph{J. Combin. Theory Ser. A}, Vol. 61, No. 2 (1992) 309--318. 

\end{thebibliography}
\bibliographystyle{plain}

\end{document}